\documentclass[11pt,twoside, reqno]{amsart}

\usepackage{amsmath}
\usepackage{amsthm}
\usepackage{amsfonts, amssymb}
\usepackage{mathrsfs}
\usepackage[all]{xy}
\usepackage{url}

\usepackage{latexsym}
\usepackage[dvips]{graphicx}

\theoremstyle{plain}
\newtheorem{lema}{Lemma}
\newtheorem{prop}[lema]{Proposition}
\newtheorem{teo}[lema]{Theorem}

\newtheorem{coro}[lema]{Corollary}
\theoremstyle{remark}

\newtheorem{obs}[lema]{Remark}

\theoremstyle{definition}
\newtheorem{defi}[lema]{Definition}
\newtheorem{ej}[lema]{Example}

\renewcommand{\sigma}{g}
\renewcommand{\varphi}{h}

\newcommand{\R}{\mathbb{R}}
\newcommand{\Z}{\mathbb{Z}}

\newcommand{\aut}{\textrm{Aut}}

\begin{document}

\title[Smallest posets with given cyclic automorphism group]{Smallest posets with given cyclic automorphism group}

\author[J.A. Barmak]{Jonathan Ariel Barmak}
\author[A.N. Barreto]{Agust\'in Nicol\'as Barreto}

\thanks{Both authors were supported by CONICET and partially supported by grant UBACyT 20020190100099BA. The first named author was also partially supported by grants CONICET PIP 11220170100357CO, ANPCyT PICT-2017-2806 and ANPCyT PICT-2019-02338.}

\address{Universidad de Buenos Aires. Facultad de Ciencias Exactas y Naturales. Departamento de Matem\'atica. Buenos Aires, Argentina.}

\address{CONICET-Universidad de Buenos Aires. Instituto de Investigaciones Matem\'aticas Luis A. Santal\'o (IMAS). Buenos Aires, Argentina. }

\email{jbarmak@dm.uba.ar}

\email{abarreto@dm.uba.ar}

\begin{abstract}
For each $n\ge 1$ we determine the minimum number of points in a poset with cyclic automorphism group of order $n$.
\end{abstract}

\makeatletter
\@namedef{subjclassname@2020}{%
  \textup{2020} Mathematics Subject Classification}
\makeatother



\maketitle

\section{Introduction}

In 1938 R. Frucht \cite{Fru} proved that any finite group can be realized as the automorphism group of a graph. Moreover, the graph can be taken with $3d|G|$ vertices, where $d$ is the cardinality of any generator set of $G$ (\cite[Theorems 3.2, 4.2]{Fru49}). In 1959 G. Sabidussi \cite{Sab} showed that in fact $O(|G|\textrm{log}(d))$ vertices suffice. In 1974 L. Babai proved that the number of generators is not relevant, and with exception of the cyclic groups $\Z_3, \Z_4$ and $\Z_5$, the graph can be taken with just $2|G|$ vertices. Sabbidussi claims in \cite{Sab} that he was able to compute the smallest number of vertices $\alpha (G)$ in a graph with automorphism group $G$ in the case that $G$ is cyclic of prime power order. Also, he asserts that for $n=p_1^{r_1}p_2^{r_2}\ldots p_k^{r_k}$, $\alpha (\Z_n)=\sum\limits_{i=1}^k \alpha(\Z_{p_i^{r_i}})$. Unfortunately both his computations for $\Z_{p^r}$ and the assertion are wrong. In \cite{Mer} R.L. Meriwether rectifies these errors and correctly determines $\alpha(\Z_n)$ for any $n\ge 1$. However, he commits similar mistakes when trying to extend this computation to arbitrary finite abelian groups. In \cite{Arl1, Arl2} W. Arlinghaus provides a complete calculation of $\alpha (G)$ for $G$ finite abelian. The proof follows these steps. First compute $\alpha (G)$ for $G$ cyclic of prime power order, then for arbitrary finite cyclic groups, then for abelian $p$-groups and finally, the general case.

In parallel, the analogous problem was studied for partially ordered sets. In 1946 G. Birkhoff \cite{Bir} proved that for any finite group $G$ there is a poset of $|G|(|G|+1)$ points and automorphism group isomorphic to $G$. Then Frucht \cite{Fru50} improved this to $(d+2)|G|$ points. In 1980 Babai \cite{Bab2} proved that $3|G|$ points are enough. However, the smallest number $\beta (G)$ of points of a poset with an arbitrary finite abelian group $G$ of automorphisms has not yet been determined. In this paper we compute $\beta (G)$ for every finite cyclic group $G$. This result was first announced in \cite{Barr}. In \cite{Barr} we computed first $\beta (G)$ for $G$ cyclic of prime power order, then for arbitrary finite cyclic and for finite abelian $p$-groups with $p\ge 11$, following the steps of the proof of the graph case exposed by Arlinghaus. The calculation of $\beta(\Z_n)$ in this paper is more direct than the original we gave in \cite{Barr}. The case of $p$-groups will not be addressed in this article. Just as in graphs, the bound $\beta (\Z_n)\le \sum\limits_{i=1}^k \beta(\Z_{p_i^{r_i}})$ holds for $n=p_1^{r_1}p_2^{r_2}\ldots p_k^{r_k}$, but not the equality, in general. For instance $\beta(\Z_{12})=\beta(3)+\beta(4)-1$.

In Section \ref{sectionexamples} we construct explicit examples which provide an upper bound for $\beta(\Z_n)$. In Section \ref{sectionlemmas} we prove some lemmas concerning the cyclic structure of a generator of $\aut (P)$. In the last section we introduce the notion of weight of a prime power in a cycle, which we use in the proof of the lower bound.

\section{Construction of the examples} \label{sectionexamples}


A poset is a set with a partial order $\le$. The elements of the underlying set of a poset are called points. All posets are assumed to be finite, that is, their underlying set is finite. If $P$ is a poset and $x,y\in P$, we write $x<y$ if $x\le y$ and $x\neq y$. We say that $y$ covers $x$ if $x<y$ and there is no $x<z<y$. The edges of $P$ are the pairs $(x,y)$ such that $y$ covers $x$. The Hasse diagram of $P$ is the digraph whose vertices are the points of $P$ and the edges are the edges of $P$. If the orientation of an arrow is not indicated in the graphical representation, we assume it points upwards. A morphism $P\to Q$ of posets is an order-preserving map, i.e. a function $f$ between the underlying sets such that for every pair $x,y\in P$ with $x\le y$ we have $f(x)\le f(y)$. If $P$ is a poset, since it is finite, an automorphism of $P$ is just a permutation of the underlying set which is a morphism. A subposet of a poset $P$ is a subset of the underlying set with the inherited order. Given an automorphism $g$ of a poset $P$, we say that a subset $A$ of the underlying set of $P$ is invariant or $g$-invariant if $g(A)=A$. In this case, $g$ induces an automorphism on the subposet with underlying set $A$.

\begin{defi}
Define $b(1)=0$, $b(2)=1$, $b(3)=b(4)=b(5)=b(7)=3$. For any other prime power $p^r$, define $b(p^r)=2$. 
\end{defi}

\begin{prop} \label{ejemplos}
Let $n=p^r$, where $p\ge 2$ is a prime and $r\ge 0$. Then there exists a poset $P$ with $b(n)n$ points and automorphism group $\aut(P)$ isomorphic to $\Z_{n}$.  
\end{prop}
\begin{proof}
For $n=1$ we take the empty poset and for $n=2$ we take the discrete poset on $2$ points. By discrete we mean an antichain, i.e. a poset of pairwise incomparable elements. If $n=3,4,5,7$ we use the well-known general construction \cite{Fru50}: $P=\Z_{n}\times \{0,1,2\}$ with the order $(i,2)>(i,1)>(i,0)<(i+1,2)$ for every $i\in \Z_{n}$. It is easy to see that such poset satisfies $\aut(P) \simeq \Z_{n}$. Suppose then that $n\ge 8$. We take two copies of $\Z_{n}$: $A=\Z_{n}=\{0,1,\ldots, n-1\}$ and $A'=\{0',1',\ldots, (n-1)'\}$. Let $S=\{0,1,2,4\}\subseteq \Z_{n}$. For $i\in A$ and $j'\in A'$ we set $i<j'$ if $j-i\in S$. Any two elements in the same copy of $\Z_n$ are not comparable (see Figure \ref{fig0124}). We will prove that the automorphism group of this poset $P$ is $\Z_{n}$. It is clear that $G=\Z_n$ acts regularly on each copy of $\Z_n$ by multiplication (addition), and this gives a faithful action $G\to \aut(P)$ on $P$. So $G$ can be seen as a subgroup of $\aut(P)$. Since each automorphism of $P$ maps $0\in A$ to another minimal element of $P$, then the order of the $\aut(P)$-orbit of $0\in P$ is $n$. If we prove that the $\aut(P)$-stabilizer of $0\in P$ is trivial, then $|\aut(P)|=n$, so $\aut(P)$ is isomorphic to $G$. Let $\varphi \in \aut(P)$ be such that $\varphi (0)=0$.

\begin{figure}[h] 
\begin{center}
\includegraphics[scale=0.45]{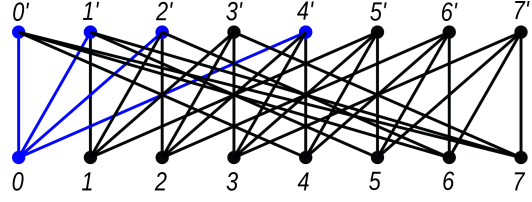}

\caption{The Hasse diagram of $P$ for $n=8$.}\label{fig0124}
\end{center}
\end{figure}

We define the \textit{double neighborhood} $B(i)$ of $i\in A$ as the set of those $j\in A$ such that $\# (P_{> i}\cap P_{> j}) \ge 2$, that is, there are at least two points in $A'$ greater than both, $i$ and $j$. The \textit{reduced double neighborhood} of $i\in A$ is $\hat{B}(i)=B(i)\smallsetminus \{i\}$.  Since $\varphi$ is an automorphism, $B(\varphi(i))=\varphi (B(i))$ and $\hat{B}(\varphi(i))=\varphi (\hat{B}(i))$. Given $k\ge 1$, we say that two points $i,j\in A$ are \textit{$k$-adjacent} if $\#(B(i)\cap B(j))=k$, and they are \textit{reduced $k$-adjacent} if $\#(\hat{B}(i)\cap \hat{B}(j))=k$. Clearly, $\varphi$ preserves $k$-adjacency and reduced $k$-adjacency. Suppose first that $n\ge 9$. Then for each $i\in A$, $B(i)=\{i-2,i-1,i,i+1,i+2\}$. It is easy to see that $i,j$ are $4$-adjacent if and only if $i-j=\pm 1$. Thus, $\varphi$ induces an automorphism of the cyclic graph on $A$ with edges given by $4$-adjacency. Since $\varphi(0)=0$, $\varphi$ is either the identity $1_{\Z_n}$ or $-1_{\Z_n}$. The second case cannot occur as $\{0,2,3,4\}$ has an upper bound while $\{0,-2,-3,-4\}$ does not. Thus every point of $A$ is fixed by $\varphi$. If $j'\in A'$, then $j'$ is the unique upper bound of $\{j,j-1,j-2,j-4\}$. Thus $\varphi(j')=j'$. This proves that $\varphi=1_P$. 

Finally, suppose $n=8$. Given $i\in A$, we have now $\hat{B}(i)=\{i-2,i-1,i+1,i+2, i+4\}$ and $i,j\in A$ are reduced $4$-adjacent if and only if $i-j=\pm 3$. Thus, $\varphi$ induces an automorphism in the cyclic graph on $A$ with edges given by reduced $4$-adjacency. Then $\varphi=1_{\Z_n}$ or $-1_{\Z_n}$. The second case cannot occur for the same reason as before. Since each point in $A'$ is determined by the set of smaller points, $\varphi=1_P$.    
\end{proof}

\begin{ej} \label{ejemplo12}
There exists a poset $P$ with $20$ points and automorphism group isomorphic to $\Z_{12}$.

Take two copies $A=\{0,1,2,3,4,5\}$, $A'=\{0',1',2',3',4',5'\}$ of $\Z_6$ and two copies $B=\{0'',1'',2'',3''\}$, $B'=\{0''',1''',2''',3'''\}$ of $\Z_4$. The underlying set of $P$ is the union of these four sets. Let $S=\{0,1,3\}\subseteq \Z_6$, $T=\{0,1\}\subseteq \Z_4$. Define the following order in $P$: $i<j'$ if $j-i\in S$, $i''<j'''$ if $j-i\in T$, $i'''<j'$ if $j-i$ is even, $i''<j$ if $j-i$ is even, $i''<j'$ for every $i,j$ (see Figure \ref{figveinte}). 

\begin{figure}[h] 
\begin{center}
\includegraphics[scale=0.4]{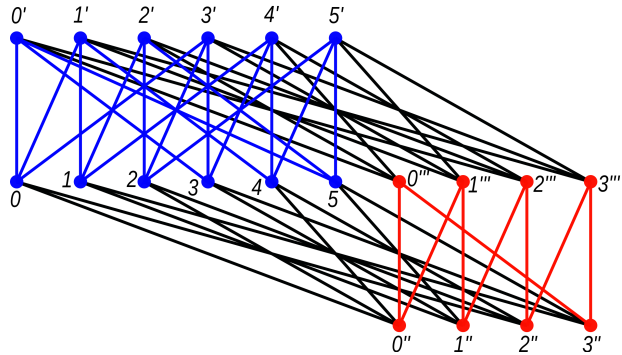}

\caption{A poset $P$ of $20$ points and $\aut(P)\simeq \Z_{12}$.}\label{figveinte}
\end{center}
\end{figure}

It is clear that $G=\Z_{12}$ acts in each copy of $\Z_6$ and of $\Z_4$ by multiplication (addition). This induces a faithful action of $G$ on $P$. If $\varphi\in \aut(P)$, $\varphi(0'')$ must be a minimal point $i''$ and $\varphi(0')$ must be a maximal point $j'$. However $i,j$ cannot have different parity. Indeed, among the points $0,2,4,0''',1'''$ which cover $0''$, there are just two $0,0'''$ smaller that $0'$. However, if $i\in \Z_4$ and $j\in \Z_6$ have different parity, among the points covering $i''$ ($k\in A$ with $k\equiv i (2)$ and $i''', (i+1)'''$) there are three smaller that $j'$: both $j-1, j-3$, and one of $i''', (i+1)'''$. Thus $i\equiv j(2)$, which implies that the $\aut(P)$-orbit of the set $\{0',0''\}$ has at most $12$ elements. If we prove that the $\aut(P)$-stabilizer of $\{0',0''\}$ is trivial, then $|\aut(P)|\le 12=|G|$, so $\aut(P)$ is isomorphic to $G$.
Let $\varphi$ be an automorphism of $P$ which fixes $0'$ and $0''$.

Note that $2''$ is the unique minimal point different from $0''$ which is covered by three points that cover $0''$. Thus $\varphi(2'')=2''$. Now, the points of $B'$ are the unique points of $P$ which cover exactly one of $0'',2''$. Thus $B'$ is invariant. This implies that $\varphi$ restricts to an automorphism of the subposet $R$ with underlying set $B\cup B'$ and of the subposet $Q$ with set $A\cup A'$.
Since $R$ is a cycle, there are only two automorphisms of $R$ fixing $0''$. One is the identity and the other maps $0'''$ to $1'''$. However, $0'''<0'$ while $1'''\nless 0'$. Thus $0'''$ is fixed by $\varphi$ and then $\varphi$ is the identity of $R$. 

Suppose that $i'\in A'$ is a fixed point. Among the points $i,i-1,i-3$ in $A$ covered by $i'$, only $i-1$ and $i-3$ share a lower bound. Thus $\varphi(i)=i$. Similarly, among the points $(i-4)',(i-2)',(i-1)'$ of $A'$ not covering $i$, only $(i-4)'$ and $(i-2)'$ share a lower bound in $B'$. Thus $(i-1)'$ is fixed. In conclusion, we showed that $i'$ fixed implies that both $i$ and $(i-1)'$ are fixed. Since $0'$ is fixed, this implies that every point of $A$ and of $A'$ is fixed. Thus $\varphi=1_P$.  

\end{ej}

We say that a prime power $p^r$ ($r\ge 1$) exactly divides an integer $n$, and write $p^r\parallel n$, if $p^r| n$ and $p^{r+1}\nmid n$.

\begin{teo} \label{teoejemplos}
Let $n=p_1^{r_1}p_2^{r_2}\ldots p_k^{r_k}$ where the $p_i$ are different primes and $r_i\ge 1$ for every $i$. Then there exists a poset with automorphism group isomorphic to $\Z_n$ and $\sum\limits_{i=1}^k b(p_i^{r_i})p_i^{r_i}-1$ points if $3\parallel n$ and $4\parallel n$, and with $\sum\limits_{i=1}^k b(p_i^{r_i})p_i^{r_i}$ points otherwise.
\end{teo}
\begin{proof}
By Proposition \ref{ejemplos}, for each $1\le i\le k$ there exists a poset $P_i$ with $b(p_i^{r_i})p_i^{r_i}$ points and $\aut(P_i)\simeq \Z_{p_i^{r_i}}$. The non-Hausdorff join or ordinal sum $P=P_1\oplus P_2\oplus \ldots \oplus P_k$ is constructed by taking a copy of each poset and keeping the given ordering in each copy, while setting $x<y$ for each $x\in P_i$ and $y\in P_j$ if $i<j$. Since each automorphism of $P$ preserves heights (the maximum length of a chain with a given maximum element), it restricts to automorphisms of each $P_i$. Thus $\aut(P)=\aut(P_1)\oplus \aut(P_2) \oplus \ldots \oplus \aut(P_k)=\Z_n$. If $p_i^{r_i}=3$ and $p_j^{r_j}=4$, instead of $P_i$ and $P_j$ we take the poset in Example \ref{ejemplo12} of $20=b(3)3+b(4)4-1$ points and automorphism group $\Z_{12}$.
\end{proof}

\section{Lemmas} \label{sectionlemmas}

Let $X$ be a finite set, $n\ge 1$ and $x_0,x_1,\ldots, x_{n-1}$ pairwise different elements of $X$. The cycle $\alpha =(x_0,x_1,\ldots , x_{n-1})$ is the permutation which maps $x_i$ to $x_{i+1}$ (indices considered modulo $n$) and fixes every other point of $X$. The number $n$ is the order or length of the cycle, which we denote by $|\alpha|$. A cycle of order $n$ is also called an $n$-cycle. A cycle $\alpha$ is non-trivial if $|\alpha|\ge 2$. The representation $(x_0,x_1,\ldots, x_{n-1})$ of a non-trivial $n$-cycle is unique up to cyclic permutation of the $n$-tuple $x_0,x_1,\ldots,x_{n-1}$. The underlying set of a non-trivial cycle $(x_0,x_1,\ldots , x_{n-1})$ is $\{x_0,x_1,\ldots , x_{n-1}\}$. Many times we will identify a non-trivial cycle with its underlying set. Two non-trivial cycles are disjoint if their underlying sets are. Any permutation $g$ of $X$ can be written as a composition $\alpha_1 \alpha_2 \ldots \alpha_k$ of disjoint non-trivial cycles. This representation is unique up to reordering of the cycles.
If a cycle $\alpha$ appears in the factorization of $g$, we say that $\alpha$ is contained in $g$ and write $\alpha \in g$. The orbits of $g$, or of the action of the cyclic group $\langle g \rangle$ on $X$, are the underlying sets of the cycles in $g$ and the singletons consisting of fixed points. Disjoint non-trivial cycles commute. Thus, if $g$ is a composition $\alpha_1 \alpha_2 \ldots \alpha_k$ of disjoint non-trivial cycles and $m\in \Z$, then $g^m=\alpha_1^m \alpha_2^m \ldots \alpha_k^m$. If $\alpha$ is a cycle of length $n$ and $m\in \Z$, the permutation $\alpha^m$ is a composition of $(n,m)=$gcd$\{n,m\}$ cycles of length $\frac{n}{(n,m)}$. In particular, $\alpha^m$ is a cycle with the same underlying set as $\alpha$ if $n$ and $m$ are coprime. Moreover, the order of $g$ is the least common multiple of the lengths of its cycles and if a cycle of $g$ has order $n$, and $m\in \Z$, then $g^m$ fixes every point of the cycle if $n|m$, and fixes no point of the cycle otherwise.

If $g$ is an automorphism of a poset $P$, then each orbit of $g$ is discrete, as $a<b$ would imply that $a<g^k(a)$ for some $k\in \Z$ and then $\{g^{nk}(a)\}_{n\ge 0}$ would be an infinite chain. If $A$ and $B$ are two different orbits of $g$ we cannot have an element $a\in A$ smaller than another $b\in B$ and at the same time an element $b'\in B$ smaller than another $a'\in A$, as this would imply that $a<b=g^k(b')<g^k(a')$ for some $k\in \Z$, contradicting the fact that $A$ is discrete, or the antisymmetry of the order.

\begin{obs} \label{extension}
Let $P$ be a poset and let $g$ be an automorphism of $P$. Let $Q$ be the subposet of points which are not fixed by $g$. Let $A_0,A_1,\ldots, A_k$ be the orbits of the automorphism induced by $g$ on $Q$. If $h$ is an automorphism of $Q$ such that $h(A_i)=A_i$ for every $i$, then it extends to an automorphism of $P$ which fixes every element not in $Q$.

Indeed, if $x\in P\smallsetminus Q$, $y\in A_i$ and $x<y$, then $h(y)\in A_i$, so there exists $r\ge 0$ such that $g^r(y)=h(y)$. Then $x=g^r(x)<g^r(y)=h(y)$. Similarly, if $x>y$, then $x>h(y)$.
\end{obs}

\begin{lema} \label{dos}
Let $n\ge 1$ and let $p^r\neq 2$ be a prime power which exactly divides $n$. Let $P$ be a poset with $\aut(P)$ cyclic of order $n$, and let $g$ be a generator of $\aut (P)$. Then $g$ contains at least two cycles of length divisible by $p^r$. 
\end{lema}
\begin{proof}
Since $g$ has order $n$, it contains at least one cycle $\alpha$ of length divisible by $p^r$. Assume there is no other cycle of length divisible by $p^r$. The automorphism $g^{\frac{n}{p}}$ fixes then every point not in $\alpha$. Let $x$ be an element of $\alpha$ and let $\tau$ be the transposition of the underlying set of $\alpha$ which permutes $x$ and $g^{\frac{n}{p}}(x)\neq x$. By Remark \ref{extension}, $\tau$ extends to an automorphism $h$ of $P$ which is a transposition. But any power of $g$ either fixes each point in $\alpha$ or fixes no point of $\alpha$. Since the order of $\alpha$ is at least $p^r>2$, $h \notin \langle g\rangle=\aut(P)$, a contradiction.  
\end{proof}

If a group $G$ acts on a poset $P$, an automorphism of $P$ is said to be induced by the action if it is in the image of the homomorphism $G\to \aut(P)$.

\begin{lema} \label{nuevofacil}
Let $p=3,5$ or $7$. Let $P$ be a poset on which $\Z_p$ acts with exactly two orbits, both of order $p$. Then there exists an automorphism of $P$ not induced by the action for which each orbit of the action is invariant.
\end{lema}
\begin{proof}
Let $g=\alpha \beta \in \aut(P)$ be the automorphism induced by a generator of $\Z_p$, where $\alpha=(0,1,\ldots, p-1)$ and $\beta=(0',1',\ldots, (p-1)')$. If no element of $\alpha$ is comparable with an element of $\beta$, then the transposition $(0,1)$ is an automorphism which is different to $g^k$ for any $k\in \Z$, that is, not induced by the action.

Without loss of generality we can assume then that $0$ and $0'$ are comparable, and moreover, that $0<0'$. Then no element in $\beta$ can be smaller than another in $\alpha$. Since $g$ is an automorphism, $i<i'$ for every $0\le i\le p-1$. If no other pair of elements are comparable, then $(0,1)(0',1')$ is an automorphism not induced by the action (it has order $2$, for example). If $i<j'$ for every $0\le i,j\le p-1$, then $(0,1)$ satisfies the desired property. This completes the proof of the case $p=3$ by the following argument. The case we did not analyze is when $P$ has exactly $6$ edges. In that case, let $P^c$ be the \textit{complement} of $P$, defined as the poset $P^c$ with the same underlying set and setting $i<j'$ if and only if $i\nless j'$ in $P$, while $i,j$ are not comparable and $i',j'$ are not comparable for every $i\neq j$. Since $P$ and $P^c$ are non-discrete, they have the same automorphisms. As $P^c$ has only $3$ edges, there is an automorphism of $P^c$ not induced by the action, so this is the required automorphism of $P$.

For $p=5$ we need to consider the case that $P$ has $10$ edges. By the complement argument, this will complete the $p=5$ case. So, suppose $0<k'$ for some $1\le k\le 4$ (and then $i<(i+k)'$ for every $i$, where $i+k$ is considered modulo $5$). Note that $g^k$ is induced by another generator of $\Z_p$ and it maps $i'$ to $(i+k)'$. Thus, for each $0\le i\le 4$, $i<i'$ and $i<g^k(i')$. Therefore we can assume that $k=1$. We have then the ``symmetry about the axis $03'$'', which maps $i$ to $-i$ and $j'$ to $(1-j)'$ (see Figure \ref{figcinco}). This is an automorphism of $P$ which is different to any power of $g$ (it has order $2$).

\begin{figure}[h] 
\begin{center}
\includegraphics[scale=0.18]{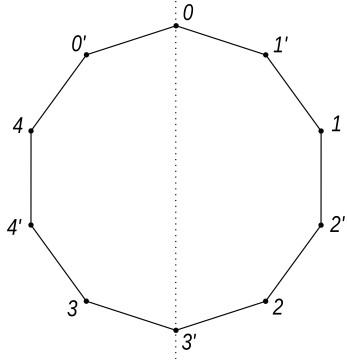}

\caption{The underlying undirected graph of a poset with $10$ points and edges $i'>i<(i+1)'$, and the axis $03'$.}\label{figcinco}
\end{center}
\end{figure}

For $p=7$, if $P$ has $14$ edges, then by the argument above we can assume $i'>i<(i+1)'$ for every $0\le i\le 6$ and there is then a symmetry about $04'$. By the complement argument it only remains to analyze the case that $P$ has exactly $21$ edges. Here $i<i',(i+k)', (i+l)'$ for certain $1\le k\neq l\le 6$ and again we can assume $k=1$ by replacing $g$ by $g^k$. Finally, by replacing $g$ by $g^{-1}$, it suffices to consider the cases $l=2,3$ and $4$ (Figure \ref{figsietetres}).

\begin{figure}[h] 
\begin{center}
\includegraphics[scale=0.4]{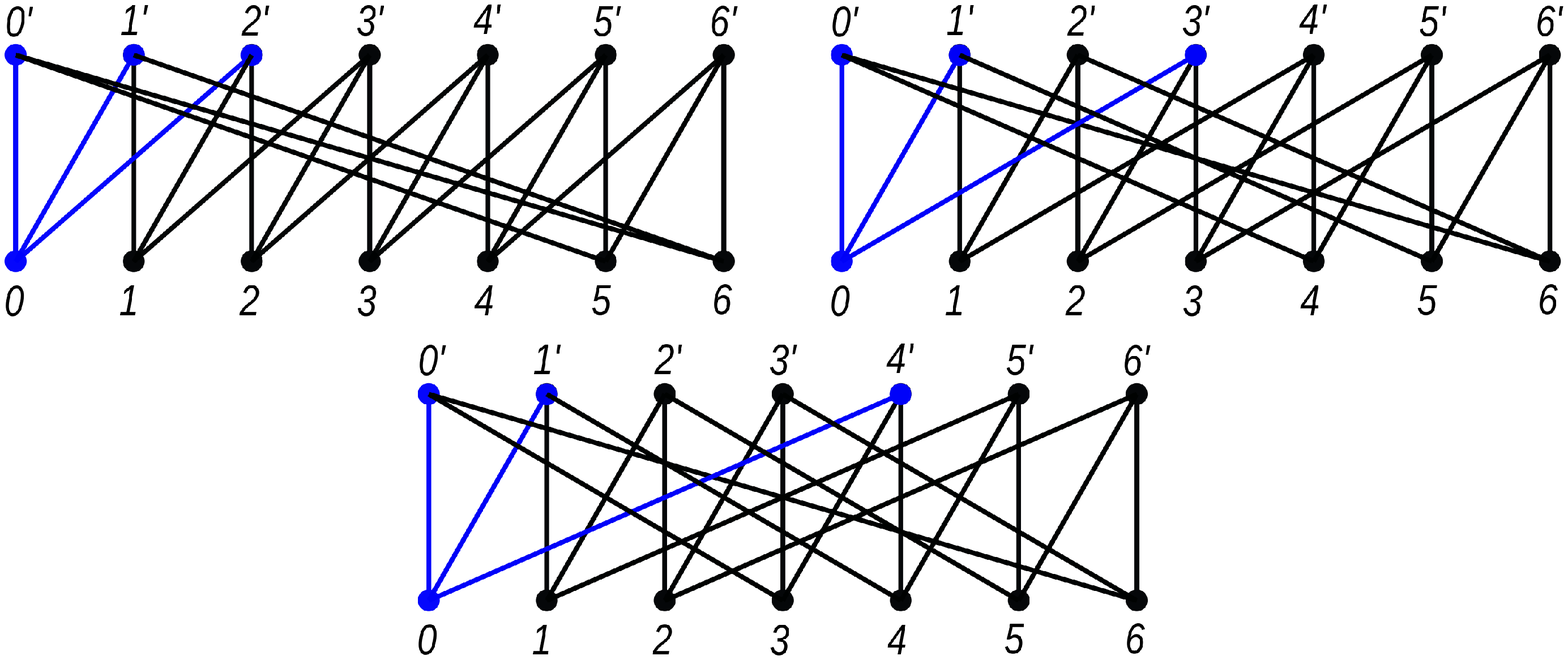}

\caption{Posets with two $\Z_7$-regular orbits and $S=\{0,1,l\}$ for $l=2,3,4$.}\label{figsietetres}
\end{center}
\end{figure}

For $l=2$ we have the involution that maps $i$ to $-i$ and $j'$ to $(2-j)'$. For $l=3$ we have the following automorphism of order $3$: $(142)(356)(0'3'1')(2'4'5')$ (see Figure \ref{triangulo}). For $l=4$, there is again the symmetry about $04'$. 
\begin{figure}[h] 
\begin{center}
\includegraphics[scale=0.2]{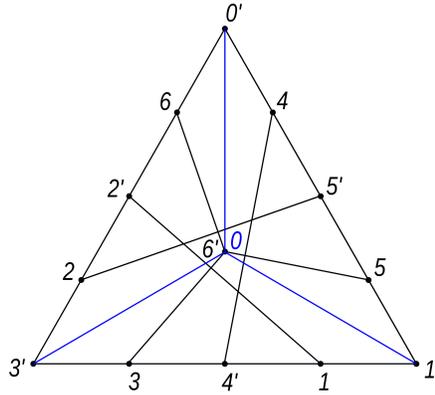}

\caption{The underlying graph of the poset $P$ of $14$ points and edges $i<i',(i+1)',(i+3)'$. An automorphism of order $3$ is given by a rotation of angle $\frac{2\pi}{3}$.}\label{triangulo}
\end{center}
\end{figure}
\end{proof}

\begin{lema} \label{nuevodificil}
Let $P$ be a poset on which $\Z_4$ acts with exactly two orbits of order $4$ or exactly three orbits: two  of order $4$ and one of order $2$. Then there exists an automorphism of $P$ not induced by the action for which each orbit of the action is invariant.
\end{lema}
\begin{proof}
Let $g$ be an automorphism induced by a generator of the action and suppose first that $g=(0,1,2,3)(0',1',2',3')$. If $P$ is discrete, $(0,1)$ satisfies the required conditions. If $P$ has exactly $4$ edges, then as in the proof of Lemma \ref{nuevofacil} we can assume $i<i'$ for every $0\le i\le 3$, and $(0,1)(0',1')$ works. By the complement argument we can assume $P$ has exactly $8$ edges and that it is determined by the relations $i'>i<(i+k)'$ for some $1\le k\le 3$. The case $k=3$ reduces to the case $k=1$ by replacing $g$ by $g^3$. If $k=1$, the symmetry $(1,3)(0',1')(2',3')$ about $02$ satisfies the required conditions. If $k=2$, then $(0,2)$ works.

Suppose then that $g=\alpha \beta \gamma$ with $\alpha=(0,1,2,3)$, $\beta=(0',1',2',3')$, $\gamma=(0'',1'')$. Let $Q$ be the subposet of points in $\alpha$ and $\beta$. Since $g^2=(0,2)(1,3)(0',2')(1',3')$, every automorphism of the poset $Q$ which has $\{0,2\},\{1,3\}, \{0',2'\}, \{1',3'\}$ as invariant sets, extends to $P$ by Remark \ref{extension}. If $Q$ is discrete or if $Q$ has $16$ edges, then $(0,2)$ is an automorphism of $Q$ which extends to $P$ and this extension is not induced by the action. If $Q$ has exactly $4$ edges, we may assume $i<i'$ for every $i$ and then $(0,2)(0',2')$ extends to an automorphism of $P$ different to any power of $g$. If $Q$ has exactly 12 edges, the complement argument can be used. Suppose then $Q$ has exactly $8$ edges. By relabelling we can assume the relations are (a) $i<j'$ for $i\equiv j (2)$ or (b) $i'>i<(i+1)'$ for every $i$. In case (a), $(0,2)$ is again an automorphism which has every nontrivial orbit of $g^2$ as an invariant set. In the rest of the proof we assume we are in case (b). 

If the points of $\gamma$ are not comparable with any point of $Q$, then the symmetry about $02$ which maps $i$ to $-i$ and $j'$ to $(1-j)'$, is an automorphism of $Q$ which extends to $P$, and this extension satisfies the required conditions.

By considering the opposite order, we can assume a point of $\gamma$ is comparable with a point of $\alpha$. Moreover, by relabelling if needed we can assume $0''$ is comparable with $0$. Suppose first that $0''<0$. Since $g$ is an automorphism, then $0''<2$ and $1''<1,3$. If $0''\nless 1$, then $0''\nless 3$ and $1''\nless 0,2$. If $0''<1$, then $0''<3$ and $1''<0,2$. In either case, the symmetry of $Q$ about $02$ extends by the identity to an automorphim of $P$ which is not induced by the action, even though this automorphism of $Q$ does not have the orbits of $g^2$ as invariant sets. Finally suppose $0''>0$. Then $0''>2$ and $1''>1,3$. We can assume no element in $\beta$ is smaller than an element in $\gamma$, by the previous case and the duality argument. Also, we cannot have an element of $\gamma$ being smaller than another $j'$ of $\beta$, since this would imply that $i<j'>i+2$, modulo 4, for certain $0\le i\le 3$, which is absurd. In any case, if $0''\ngtr 1$ or if $0''>1$, we have that the symmetry of $Q$ about $02$ extends to an automorphism of $P$. 
\end{proof}

\begin{lema} \label{lema357}
Let $p=3,5$ or $7$. Let $P$ be a poset with cyclic automorphism group of order $n\ge 1$, and let $g\in \aut(P)$ be a generator. Suppose $g$ contains a $p$-cycle $\alpha$ and a $pk$-cycle $\beta\neq \alpha$ for some $p\nmid k\ge 1$. Then it contains a third cycle whose length is divisible by $p$.
\end{lema}
\begin{proof}
Suppose $\beta=(0,1,\ldots, pk-1)$. Let $Q$ be the subposet of $P$ whose points are those of $\alpha$ and $\beta$. Assume that there is no other cycle in $g$ whose length is divisible by $p$. In particular $p\parallel n$. Since the order of any cycle of $g$ different from $\alpha$ and $\beta$ divides $\frac{n}{p}$, the automorphism $g^{\frac{n}{p}}$ fixes every point not in $Q$. Moreover $g^{\frac{n}{p}}$ has $k+1$ orbits of order $p$, which are the underlying set of $\alpha$ and $A_i=\{0\le j\le pk-1 | \ j\equiv i (k)\}$ for $0\le i\le k-1$. In particular, by Remark \ref{extension} every automorphism of $Q$ for which these sets are invariant extends to an automorphism of $P$.

Let $Q'$ be the subposet of $Q$ whose points are those of $\alpha$ and $A_0$. Since $g^{k}$ induces an automorphism of $Q'$ with two orbits of order $p$, by Lemma \ref{nuevofacil} there is an automorphism $\varphi$ of $Q'$ not induced by a power of $g^{k}$ for which the underlying set of $\alpha$ and $A_0$ are invariant. We extend $\varphi$ to an automorphism $\overline{\varphi}$ of $Q$ as follows. Let $j$ be a point of $\beta$, $0\le j\le kp-1$. Let $0\le i\le k-1$ be such that $j\in A_i$. Since $p\nmid k$, there exists a unique  $0\le t\le k-1$ such that $k|j+tp$, in other words $j+tp$, considered modulo $kp$, lies in $A_0$. Then $\varphi(j+tp)\in A_0$. Define $\overline{\varphi}(j)=\varphi(j+tp)-tp\in A_i$. We claim that $\overline{\varphi}$ is an automorphism of $Q$. It is clearly bijective. Two different points of $\beta$ cannot be comparable as they are in the same orbit. Suppose $j$ in $\beta$ and $a$ in $\alpha$ are comparable, say $a<j$. Let $0\le t\le k-1$ be such that $k|j+tp$. Then $a=g^{tp}(a)<g^{tp}(j)=j+tp$. Since $\varphi$ is a morphism, $\varphi(a)<\varphi(j+tp)$. Thus $\overline{\varphi}(a)=\varphi (a)=g^{-tp}(\varphi(a))<g^{-tp} (\varphi(j+tp))=\varphi(j+tp)-tp=\overline{\varphi}(j)$. Since the underlying set of $\alpha$ and each $A_i$ are $\overline{\varphi}$-invariant, $\overline{\varphi}$ extends to an automorphism of $P$, which must be a power $g^r$ of $g$. Since $g^r$ leaves $A_0$ invariant, in particular $r=g^r(0)\in A_0$, so $k|r$ and $\varphi$ is then induced by a power of $g^k$, a contradiction.   
\end{proof}

\begin{lema} \label{lema4}
Let $P$ be a poset with cyclic automorphism group of order $n\ge 1$, and let $g\in \aut (P)$ be a generator. Suppose that $g$ contains two $4$-cycles $\alpha, \beta$. Then it contains a third cycle of length divisible by $4$ or two more cycles of even length. 
\end{lema}
\begin{proof}
The proof is very similar to that of Lemma \ref{lema357}, so we omit details. If $\alpha$ and $\beta$ are the unique two cycles of even length in $g$, then by Lemma \ref{nuevodificil} there is an automorphism $h$ of the poset of points of these two cycles which is not induced by a power of $g$, and moreover has the underlying sets of $\alpha$ and $\beta$ as invariant sets. Since the non-trivial orbits of $g^{\frac{n}{4}}\in \aut(P)$ are the underlying sets of $\alpha$ and $\beta$, $h$ extends to an automorphism of $P$, a contradiction.  

Suppose then there exists a third cycle $\gamma=(1,2,\ldots, 2k)$ in $g$ with $k$ odd, and that there is no other cycle of even length. We define $Q$ to be the subposet whose points are those of $\alpha, \beta$ and $\gamma$. Then $g^{\frac{n}{4}}$ fixes every point not in $Q$. The other orbits of $g^{\frac{n}{4}}$ are the underlying sets of $\alpha$ and $\beta$, and $A_i=\{i,k+i\}$ for $0\le i\le k-1$. Let $Q'$ be the subposet whose points are those of $\alpha, \beta$ and $A_0$. Then $g^k$ induces an automorphism of $Q'$ and by Lemma \ref{nuevodificil} there is an automorphism $\varphi$ of $Q'$ which is not induced by a power of $g^k$, and for which the underlying sets of $\alpha, \beta$ and $A_0$ are invariant. We extend it to an automorphism $\overline{\varphi}$ of $Q$ by defining $\overline{\varphi}(j)=\varphi(j+4t)-4t$, where $t$ is such that $k|j+4t$. Then $\overline{\varphi}$ is bijective, it is a morphism and leaves each $A_i$ invariant. It extends to an automorphism of $P$, say $g^r$. Since $g^r$ leaves $A_0$ invariant, then $k|r$, which implies that $h$ is induced by a power of $g^k$, a contradiction.
\end{proof}

\section{Weights and the lower bound}  \label{sectionweights}

Let $\sigma$ be a permutation of order $n$ of a finite set $X$. Let $\alpha$ be a cycle in $\sigma$ of length $l=p_1^{r_1}p_2^{r_2}\ldots p_k^{r_k}$, where the $p_i$ are distinct prime integers, $r_i\ge 1$ for every $i$. For each prime power $p^r$ we will define a weight $w_{p^r}(\alpha)\in \R_{\ge 0}$ which depends on $p^r,l$ and $n$, in such a way that $\sum\limits_{p^r} w_{p^r}(\alpha)p^r= l$, where the sum is taken over all prime powers dividing $n$. In particular $\# X\ge \sum\limits_{p^r\parallel n} (\sum\limits_{\alpha \in \sigma} w_{p^r}(\alpha))p^r$. For each $l\ge 2$ we will assign the weight of every prime power $p^r$ in a cycle $\alpha$ of length $|\alpha|=l$ according to a series of rules. In every case, if the weight $w_{p^r}(\alpha)$ is not explicitly defined for some prime power, we assume it is $0$.  

\medskip

\noindent \textbf{Exception 6}. Suppose $l=6$. If $3\parallel n$ then $w_3(\alpha)=2$. If $3\nparallel n$ and $2\parallel n$, then $w_{2}(\alpha)=3$. If $3\nparallel n$ and $2\nparallel n$, then $w_4(\alpha)=\frac{3}{2}$.

\medskip

\noindent \textbf{Exception 12}. Suppose $l=12$. If $3\parallel n$ then $w_3(\alpha)=4$. If $3\nparallel n$, then $w_{4}(\alpha)=3$.

\medskip

\noindent \textbf{Exception 10-14}. Suppose $l=2p$ for $p=5$ or $7$. If $2\parallel n$, $w_2(\alpha)=1$. Otherwise $w_4(\alpha)=\frac{1}{2}$. In any case $w_p(\alpha)=\frac{2(p-1)}{p}$.



\medskip

\noindent \textbf{General case}. Suppose  $l=p_1^{r_1}p_2^{r_2}\ldots p_k^{r_k}\neq 6,12,10,14$, where the $p_i$ are different primes and each $r_i\ge 1$. For each $1\le i\le k$, we define $w_{p_i^{r_i}}(\alpha)=\frac{\prod\limits_{j\neq i} p_j^{r_j}}{k}$, unless $p_i^{r_i}=2$ and $2\nparallel n$. In that case, $w_2(\alpha)=0$, while $w_4(\alpha)=\frac{\prod\limits_{j\neq i} p_j^{r_j}}{2k}$. In particular, if $l=p^r\ge 3$ is a prime power, $w_{p^r}(\alpha)=1$.

\bigskip

Note that, as we required, the sum $\sum\limits_{p^r|n} w_{p^r}(\alpha)$ over all the prime powers dividing $n$ is the length $l$ of $\alpha$. Note also that if $l=p_1^{r_1}p_2^{r_2}\ldots p_k^{r_k}$, then $w_{p^r}(\alpha)\neq 0$ only if $p^r=p_i^{r_i}$ for some $1\le i\le k$ or $p^r=4$. 

\begin{teo} \label{main}
Let $n\ge 1$. Let $P$ be a poset with $\aut (P)$ cyclic of order $n$ generated by $g$. Let $p^r$ be a prime power which exactly divides $n$. If $p^r\neq 2,4$ then $\sum\limits_{\alpha \in \sigma} w_{p^r}(\alpha)\ge b(p^r)$. 
If $3\nparallel n$ and $p^r=2$ or $p^r=4$, $\sum\limits_{\alpha \in \sigma} w_{p^r}(\alpha)\ge b(p^r)$ as well.
If $3\parallel n$ and $2\parallel n$, $\sum\limits_{\alpha \in \sigma} (2w_{2}(\alpha)+3w_{3}(\alpha))\ge 2b(2)+3b(3)=11$. Finally, if $3\parallel n$ and $4\parallel n$, $\sum\limits_{\alpha \in \sigma} (4w_{4}(\alpha)+3w_{3}(\alpha))\ge 4b(4)+3b(3)-1=20$.
\end{teo}
\begin{proof}
If $p^r\neq 2,3,4,5,7$, by Lemma \ref{dos}, there are at least two cycles of length divisible by $p^r$. By hypothesis their lengths are not multiples of $p^{r+1}$. But if $\alpha$ is a cycle of $\sigma$ whose length is a multiple of $p^r$, then $w_{p^r}(\alpha)\ge 1$. Indeed, the weights in $\alpha$ are assigned according to the General case. If the length of $\alpha$ is $l=p_1^{r_1}p_2^{r_2}\ldots p_k^{r_k}$, we can assume $p^r=p_1^{r_1}$ and then $w_{p^r}(\alpha)=\frac{\prod\limits_{j=2}^k p_j^{r_j}}{k}\ge \frac{2^{k-1}}{k}\ge 1$. Thus, $\sum\limits_{\alpha \in \sigma} w_{p^r}(\alpha)\ge 2= b(p^r)$.

Suppose now $p^r=5$. If $\alpha$ is a cycle of $\sigma$ of length $l=5$, then $w_5(\alpha)=1$. If $l=10$, then $w_5(\alpha)=\frac{8}{5}\ge \frac{3}{2}$ (Exception 10-14). If $l=5s$ with $s=p_2^{r_2}p_3^{r_3}\ldots p_k^{r_k}\ge 3$ not divisible by $5$, then either $k=2$, or $k\ge 3$. In the first case $w_5(\alpha)=\frac{s}{2}\ge \frac{3}{2}$, and in the second case $w_5(\alpha)=\frac{\prod\limits_{j=2}^k p_j^{r_j}}{k}\ge \frac{2^{k-2}.3}{k}\ge 2\ge \frac{3}{2}$. 

By Lemma \ref{dos}, there are at least two cycles of length divisible by $5$ (and not by $5^2$). Suppose first there exactly two such cycles, $\alpha$ and $\alpha'$. None of them can be of length $5$ by Lemma \ref{lema357}. Thus $w_5(\alpha)+w_5(\alpha')\ge 2. \frac{3}{2}= 3=b(5)$. Finally, if there are at least three cycles in $\sigma$ of length divisible by $5$, then $\sum\limits_{\alpha \in \sigma} w_{5}(\alpha)\ge 3= b(5)$.

The case $p^r=7$ is similar to the previous one, with the observation that for length $l=14$, $w_7(\alpha)=\frac{12}{7}\ge \frac{3}{2}$ (Exception 10-14). So, also in this case $\sum\limits_{\alpha \in \sigma} w_{7}(\alpha)\ge 3= b(7)$.

Let $p^r=3$. If the length of a cycle $\alpha$ in $g$ is $l=3$, $w_3(\alpha)=1$. If $l=6$, $w_3(\alpha)=2$ (Exception 6). If $l=12$, $w_3(\alpha)=4$ (Exception 12). If $l=3s$ with $s=p_2^{r_2}p_3^{r_3}\ldots p_k^{r_k}\ge 5$, then either $k=2$, or $k\ge 3$. In the first case $w_3(\alpha)=\frac{s}{2}\ge \frac{5}{2}$, and in the second case $w_3(\alpha)=\frac{\prod\limits_{j=2}^k p_j^{r_j}}{k}\ge \frac{2^{k-2}.3}{k}\ge 2$. 

By Lemma \ref{dos} there are at least two cycles in $\sigma$ of length divisible by $3$ (and not by $3^2$). Suppose first there are exactly two such cycles $\alpha$ and $\alpha'$. None of them can have length $3$ by Lemma \ref{lema357}. Then $w_3(\alpha)+w_3(\alpha')\ge 2.2=4\ge 3=b(3)$. Finally, if there are at least three cycles in $\sigma$ of length divisible by $3$, then $\sum\limits_{\alpha \in \sigma} w_{3}(\alpha)\ge 3= b(3)$. Note that $\sum\limits_{\alpha \in \sigma} w_{3}(\alpha)\ge 4$ unless there are exactly three cycles of length $3$ and no other cycle of length divisible by $3$.

We analyze now the case that $3\nparallel n$ and $p^r=2$ or $4$. In the first situation, there is at least one cycle $\alpha$ of even length $l$ (not divisible by $4$). If $l=2$, $w_2(\alpha)=1$ (General case). If $l=6$, $w_2(\alpha)=3$ (Exception 6). If $l=10$ or $l=14$, then $w_2(\alpha)=1$ (Exception 10-14). If $l=2s$ with $s=p_2^{r_2}p_3^{r_3}\ldots p_k^{r_k}\neq 1,3,5,7$ (odd), then $w_2(\alpha)=\frac{\prod\limits_{j=2}^k p_j^{r_j}}{k}\ge \frac{3^{k-1}}{k}\ge \frac{3}{2}$. Thus $\sum\limits_{\alpha \in \sigma} w_{2}(\alpha)\ge 1= b(2)$. We consider the second situation, $p^r=4$. If $\alpha$ has length $l=4$, then $w_4(\alpha)=1$. If $l=12$, $w_4(\alpha)=3$ (Exception 12). If $l=4s$ with $s=p_2^{r_2}p_3^{r_3}\ldots p_k^{r_k}\ge 5$ (odd), then $k=2$ or $k\ge 3$. For $k=2$ we have $w_4(\alpha)=\frac{s}{2}\ge \frac{5}{2}$. For $k\ge 3$, $w_4(\alpha)\ge \frac{3^{k-1}}{k}\ge 3$. By Lemma \ref{dos}, $\sigma$ contains at least two cycles of lengths divisible by $4$ (and not by $8$). Suppose first there are exactly two such cycles, $\alpha$ and $\alpha'$, of lengths $l,l'$. If $l=l'=4$, then by Lemma \ref{lema4}, there exists a third and a fourth cycle $\beta, \beta'$ of lengths $2m$ and $2m'$ for some odd $m,m'$. The weights $w_4(\beta)$ that we obtain for each $m$ are the halves of the weights that we obtained for $2$ in cycles of the same length when $2\parallel n$. Namely, if $m=1$, $w_4(\beta)=\frac{1}{2}$ (General case); if $m=3$, $w_4(\beta)=\frac{3}{2}$ (Exception 6); if $m=5,7$, $w_4(\beta)=\frac{1}{2}$ (Exception 10-14); if $m=p_2^{r_2}p_3^{r_3}\ldots p_k^{r_k}\neq 1,3,5,7$ then $w_4(\beta)\ge \frac{3^{k-1}}{2k}\ge \frac{3}{4}$ (General case).

The same happens with $\beta'$. Thus $w_4(\alpha)+w_4(\alpha')+w_4(\beta)+w_4(\beta')\ge 1+1+\frac{1}{2}+\frac{1}{2}=3=b(4)$. If instead $l=4$ and $l'=12$, then $w_4(\alpha)+w_4(\alpha')=1+3=4> 3$. If $l=4$ and $l'=4s$ for some odd $s\ge 5$, then $w_4(\alpha)+w_4(\alpha')\ge 1+\frac{5}{2}>3$. If both $l$ and $l'$ are greater than $4$, then $w_4(\alpha)+w_4(\alpha')\ge \frac{5}{2}+\frac{5}{2}>3$. Finally, if there are at least three cycles of length divisible by $4$, then $\sum\limits_{\alpha \in \sigma} w_4(\alpha)\ge 3$. Thus, in any case $\sum\limits_{\alpha \in \sigma} w_4(\alpha)\ge 3=b(4)$.

It only remains to analyze the case $3\parallel n$ and $2\parallel n$ and the case $3\parallel n$ and $4\parallel n$.        
If $3\parallel n$ and $2\parallel n$, recall that we have already proved that $\sum\limits_{\alpha \in \sigma} w_{3}(\alpha)\ge 4$ or there are exactly three cycles of length $3$ and no other cycle of length divisible by $3$. In the first case $\sum\limits_{\alpha \in \sigma} (2w_{2}(\alpha)+3w_{3}(\alpha))\ge \sum\limits_{\alpha \in \sigma} 3w_{3}(\alpha)\ge 12$. In the second case, there exists a cycle $\beta$ in $\sigma$ of even length $m\neq 6$, so $w_2(\beta)\ge 1$. Thus $\sum\limits_{\alpha \in \sigma} (2w_{2}(\alpha)+3w_{3}(\alpha))\ge 2.1+3.3= 11$. 

The last case is $3\parallel n$ and $4\parallel n$.
Note that if there are no cycles of length $6$ nor $12$ in $\sigma$, then the computation $\sum\limits_{\alpha \in \sigma} w_4(\alpha)\ge 3$ remains valid as Exceptions 6 and 12 do not occur. Thus $\sum\limits_{\alpha \in \sigma} (3w_{3}(\alpha)+4w_{4}(\alpha))\ge 3.3+4.3=21>20$. If there are at least two $12$-cycles, then $\sum\limits_{\alpha \in \sigma} (3w_{3}(\alpha)+4w_{4}(\alpha))\ge 2.3.4=24>20$. If there is no $12$-cycle in $\sigma$ and $\sum\limits_{\alpha \in \sigma} w_4(\alpha)< 3$, then we must be in the case that there is a $6$-cycle. This already implies $\sum\limits_{\alpha \in \sigma} w_3(\alpha)\ge 4$, while the existence of two cycles of length divisible by $4$ implies $\sum\limits_{\alpha \in \sigma} w_4(\alpha)\ge 2$. Thus $\sum\limits_{\alpha \in \sigma} (3w_{3}(\alpha)+4w_{4}(\alpha))\ge 3.4+4.2=20$.

Thus we may assume $\sigma$ has a unique $12$-cycle. By Lemma \ref{dos} there is another cycle of length divisible by $4$, so $\sum\limits_{\alpha \in \sigma} w_4(\alpha)\ge 1$. On the other hand, $\sum\limits_{\alpha \in \sigma} w_{3}(\alpha)\ge 4+2=6$, as the weight of $3$ in a $12$-cycle is $4$ and by Lemmas \ref{dos} and \ref{lema357} there are either two more cycles of lengths divisible by $3$ or just one, but of length not $3$. Thus $\sum\limits_{\alpha \in \sigma} (3w_{3}(\alpha)+4w_{4}(\alpha))\ge 3.6+4.1=22>20$.
\end{proof}

\begin{coro}
Let $n=p_1^{r_1}p_2^{r_2}\ldots p_k^{r_k}$, where the $p_i$ are different primes and $r_i\ge 1$ for every $i$. Then the minimum number $\beta(\Z_n)$ of points in a poset with cyclic automorphism group of order $n$ is $\sum\limits_{i=1}^k b(p_i^{r_i})p_i^{r_i}-1$ if $3\parallel n$ and $4\parallel n$, and $\sum\limits_{i=1}^k b(p_i^{r_i})p_i^{r_i}$ otherwise.
\end{coro}
\begin{proof}
If $P$ is a poset with $\aut (P)\simeq \Z_n$ generated by $g$, then the number of points in $P$ is at least $\sum\limits_{\alpha \in g} |\alpha|=$ $\sum\limits_{\alpha \in g} \sum\limits_{p^r|n}w_{p^r}(\alpha)p^r\ge \sum\limits_{i=1}^k (\sum\limits_{\alpha \in g} w_{p_i^{r_i}}(\alpha))p_i^{r_i}$. If both $3$ and $4$ exactly divide $n$, by Theorem \ref{main} this is $\sum\limits_{p_i^{r_i}\neq 3,4} (\sum\limits_{\alpha \in g} w_{p_i^{r_i}}(\alpha))p_i^{r_i}+\sum\limits_{\alpha \in g} (3w_{3}(\alpha)+4w_4(\alpha))\ge \sum\limits_{p_i^{r_i}\neq 3,4} b(p_i^{r_i})p_i^{r_i}+3b(3)+4b(4)-1=\sum\limits_{i=1}^k b(p_i^{r_i})p_i^{r_i}-1$. Otherwise, the bound is one more than this number. The bound is attained by Theorem \ref{teoejemplos}.
\end{proof}

\end{document}